\documentclass[a4paper,11pt,reqno]{amsart}
\usepackage[latin1]{inputenc}
\usepackage{mathrsfs}
\usepackage{dsfont}
\usepackage{hyperref}
\usepackage{amsmath}
\usepackage{amssymb}
\usepackage{amsthm}
\usepackage{amsfonts}
\usepackage{amstext}
\usepackage{amsopn}
\usepackage{amsxtra}
\usepackage{mathrsfs}
\usepackage{dsfont}
\usepackage{esint}
\usepackage{enumitem}
\usepackage{tikz-cd}

\newtheorem{thm}{Theorem}
\newtheorem{lemma}[thm]{Lemma}
\newtheorem{prop}[thm]{Proposition}

\newtheorem{remark}[thm]{Remark}

\newcommand{\R}{\mathbb{R}}

\newcommand{\C}{\mathbb{C}}

\renewcommand\phi{\varphi}

\newcommand{\cS}{\mathcal{S}}

\newcommand{\cN}{\mathcal{N}}
\newcommand{\cD}{\mathcal{D}}

\newcommand{\cH}{\mathcal{H}}
\newcommand{\cL}{\mathcal{L}}

\renewcommand{\geq}{\geqslant}
\renewcommand{\leq}{\leqslant}

\renewcommand{\tilde}{\widetilde}

\newcommand{\be}{\begin{equation}}
\newcommand{\ee}{\end{equation}}
\newcommand{\bq}{\begin{equation}}
\newcommand{\eq}{\end{equation}}

\newcommand{\eps}{\varepsilon}
\usepackage{color}

\title[Weakly localized states for nonlinear Dirac equations]{Weakly localized states for nonlinear Dirac equations}

\author[W. Borrelli]{William Borrelli}
\address{Universit\'e Paris-Dauphine, PSL Research University, CNRS, UMR 7534, CEREMADE, F-75016 Paris, France} 
\email{borrelli@ceremade.dauphine.fr}

\date{\today}

\DeclareRobustCommand{\SkipTocEntry}[5]{}

\begin{document}

\begin{abstract}
We prove the existence of infinitely many non square-integrable stationary solutions for a family of massless Dirac equations in 2D. They appear as effective equations in two dimensional honeycomb structures. We give a direct existence proof thanks to a particular radial ansatz, which also allows to provide the exact asymptotic behavior of spinor components. Moreover, those solutions admit a variational characterization as \emph{least action} critical points of a suitable action functional. We also indicate how the content of the present paper allows to extend our previous results for the massive case \cite{shooting} to more general nonlinearities.
\end{abstract}

\maketitle

\tableofcontents

\medskip

\medskip


\section{Introduction}
The Dirac equation has been widely employed to build relativistic models of particles (see, e.g., the survey paper \cite{els}). Recently, it made its appearance in Condensed Matter Physics in connection with two-dimensional honeycomb structures. They appear, for instance, in the study of new two-dimensional materials possessing Dirac fermions low-energy excitations, the most famous being graphene (see, e.g. \cite{graphene},\cite{introdiracmaterials}). In those materials electrons at the Fermi level have zero apparent mass and can be described using the massless Dirac equation. The laser beam propagation in certain photonic crystals can also be described by honeycomb structures \cite{photoniclattice}. 

More generally, Schr\"{o}dinger operators of the form
$$H=-\Delta+V_{per}(x),\quad x\in\R^{2},$$ 

where $V_{per}\in C^{\infty}(\R^{2},\R)$ is a honeycomb potential (in the sense of \cite{FWhoneycomb}), that is possessing the symmetries to a honeycomb lattice $\Lambda$, generically exhibit conical intersections (the so-called \textit{Dirac points}) in their dispersion bands, as proved in \cite{FWhoneycomb}. This leads to the appearance of Dirac as the effective operator, describing the electron dynamics for wave packets spectrally concentrated around such conical degeneracies. 

More precisely, by the \textit{Bloch-Floquet theory} \cite{reedsimonIV}, the spectrum $\operatorname{Spec}(H)\subseteq\R$ is the union of spectral bands, obtained through the following $k$-pseudoperiodic eigenvalue problem:
\begin{equation}\label{pseudo}
\left\{\begin{aligned}
   H\Phi(y,k)=\mu(k)\Phi(y,k),\quad y\in Y \\
     \Phi(y+v,k)=e^{ik\cdot v}\Phi(y,k),\quad v\in\Lambda
\end{aligned}\right.
\end{equation}
where $k\in Y^{*}$ varies in the \textit{Brillouin zone}, that is the fundamental zone of the dual lattice $\Lambda^{*}$, and $Y\subseteq \Lambda$ is the fundamental zone of the physical lattice.

The eigenvalues given by (\ref{pseudo}) form a sequence
$$\mu_{0}(k)\leq\mu_{1}(k)\leq\mu_{2}(k)\leq ... $$
and the corresponding pseudo-periodic eigenfunctions $\Phi_{n}(\cdot,k)$, called \textit{Bloch functions}. In \cite{FWhoneycomb} Fefferman and Weinstein proved that if $k=K_{*}$ is a vertex of the Brillouin zone, then there exists $m\in\mathbb{N}$ such that the bands $\mu_{n}(\cdot), \mu_{n+1}(\cdot)$ intersect conically at $\mu_{*}:=\mu_{n}(K_{*})=\mu_{n+1}(K_{*})$ and
\be\label{bloch}
\ker(H-\mu_{*})=\operatorname{span}\{\Phi_{1}(x),\Phi_{2}(x)\}.
\ee

More details and precise definitions can be found in \cite{FWhoneycomb}.
 \begin{remark}
Physically, the energy $\mu_{*}$ of the Dirac point is the Fermi level, in the case of graphene. In turns, this corresponds to the zero-energy for the Dirac operator. However, there is no interpretation of the Dirac spectrum in terms of particles/antiparticles. Rather, the positive part of the spectrum corresponds to massive conduction electrons, while the negative one to valence electrons.
\end{remark}
 
The following nonlinear Schr\"{o}dinger/Gross-Pitaevskii (NLS/GP) equation.

\be\label{gp}
i\partial_{t}u=-\Delta u+V(x)u+\kappa\vert u\vert^{2}u,\qquad x\in\R^{2},\kappa\in\R, 
\ee
 describes, in the quantum setting, the dynamics of Bose-Enstein condensates, and $u$ is the wavefunction of the condensate \cite{boseeinstein,derivation}. Here $V(x)$ models a magnetic trap and the nonlinear potential $\kappa\vert u\vert^{2}$ describes a mean-field interaction between particles. The parameter $\kappa$ is the microscopic 2-body scattering length. Another important field of application of NLS/GP is nonlinear optics, namely in the description of electromagnetic interference of beams in photorefractive crystals \cite{nonlinearoptics}. In this case $V(x)$ is determined by the spatial variations of the background linear refractive index of the medium, while the nonlinear potential accounts for the fact that regions of higher electric field intensity have a higher refractive index (the so-called Kerr nonlinear effect). In this case $\kappa<0$ represents the Kerr nonlinearity coefficient. In the latter situation, the variable $t\in\R$ denotes the distance along the direction of propagation and $x\in\R^{2}$ the transverse dimensions.

In the above systems honeycomb structures can be realized and tuned through suitable optical induction techniques based on laser or light beam interference \cite{photoniclattice}. They are encoded in the properties of the periodic potential $V=V_{per}$.

Let $u_{0}(x)=u^{\varepsilon}_{0}(x)$ be a wave packet spectrally concentrated around a Dirac point, that is: 

\be\label{concentrated}
u^{\varepsilon}_{0}(x)=\sqrt{\varepsilon}(\alpha_{0,1}(\varepsilon x)\Phi_{1}(x)+\alpha_{0,2}(\varepsilon x)\Phi_{2}(x))
\ee
where $\Phi_{j}, j=1,2$ are Bloch functions at the Dirac point (\ref{bloch}), and the functions $\alpha_{0,j}$ are some (complex) amplitudes to be determined.

Then one expects the solution to (\ref{gp}) to evolve, at leading order in $\epsilon$, as a modulation of Bloch functions:  

\be\label{approximate}
u^{\varepsilon}(t,x)\underset{\epsilon\rightarrow0^{+}}{\sim}\sqrt{\varepsilon}\left(\alpha_{1}(\varepsilon t,\varepsilon x)\Phi_{1}(x)+\alpha_{2}(\varepsilon t,\varepsilon x)\Phi_{2}(x) + \mathcal{O}(\varepsilon)\right),
\ee 

As explained in \cite{arbunichsparber}, performing a multiscale expansion one obtains the following effective Dirac system for the modulation coefficients $\alpha_{j}$:

\begin{equation}\label{effective}
\left\{\begin{aligned}
    \partial_{t}\alpha_{1}+\overline{\lambda}(\partial_{x_{1}}+i\partial_{x_{2}})\alpha_{2} &=-i\kappa(2\beta_{2}\vert\alpha_{1}\vert^{2}+\beta_{1}\vert\alpha_{2}\vert^{2})\alpha_{1} \\
     \partial_{t}\alpha_{2}+\lambda(\partial_{x_{1}}-i\partial_{x_{2}})\alpha_{1} &=-i\kappa(\beta_{1}\vert\alpha_{1}\vert^{2}+2\beta_{2}\vert\alpha_{2}\vert^{2})\alpha_{2} 
\end{aligned}\right.,
\end{equation}
where 
\be\label{beta}
\beta_{2}:=\int_{Y}\vert\Phi_{1}(x)\vert^{2}\vert\Phi_{2}(x)\vert^{2}dx,\quad\beta_{1}:=\int_{Y}\vert\Phi_{1}(x)\vert^{4}dx=\int_{Y}\vert\Phi_{2}(x)\vert^{4}dx.
\ee
Here $\lambda\in\C\setminus\{0\}$ is a coefficient related to the potential $V$ (see \cite{FWwaves},\cite{FWhoneycomb}). 

We remark that the system \eqref{effective} has been first formally derived by Fefferman and Weinstein in \cite{wavedirac}.

The large, but finite, time-scale validity of the Dirac approximation has been proved in \cite{FWwaves} for the linear case and for Schwartz class intial data (\ref{concentrated}). The case of cubic nonlinearities (\ref{effective}), corresponding to the NLS/GP (\ref{gp}), is treated in \cite{arbunichsparber} for high enough Sobolev regularity $H^{s}(\R^{2})$, with $s>3$.

\begin{remark}
The factor $\sqrt{\epsilon}$ appearing in (\ref{concentrated},\ref{approximate}) is of course irrelevant in the linear case, but it is exactly the critical scaling such that in the cubic case the nonlinearity and the Dirac appear together at first order in the multiscale expansion, as shown in \cite{arbunichsparber}.
\end{remark}

We are interested in studying zero-modes of (\ref{effective}) in the \textit{focusing case} $\kappa <0$, that is, we look for particular stationary solutions of the form
$$\alpha(t,x)=\psi(x),\qquad (t,x)\in\R\times\R^{2}. $$
It will turn out that they are in general weakly localized, as they are not even square-integrable, in contrast to the results mentioned for the evolution problem. We expect those zero-modes to be useful to prove approximation results for stationary solutions to (\ref{gp}), analogous to the ones proved in \cite{FWwaves},\cite{arbunichsparber} for the evolution problem, somehow in the spirit of \cite{ilanweinstein}. However, the absence of a gap at the Dirac point is a serious problem to deal with. This will be the object of a future investigation and will be addressed elsewhere.

It is not restrictive to choose $$ \kappa=-1$$ in (\ref{effective}). This leads to the following system:
\begin{equation}\label{}
\left\{\begin{aligned}
        \overline{\lambda}(\partial_{x_{1}}+i\partial_{x_{2}})\psi_{2} &=i(2\beta_{2}\vert\psi_{1}\vert^{2}+\beta_{1}\vert\psi_{2}\vert^{2})\psi_{1} \\
    \lambda(\partial_{x_{1}}-i\partial_{x_{2}})\psi_{1} &=i(\beta_{1}\vert\psi_{1}\vert^{2}+2\beta_{2}\vert\psi_{2}\vert^{2})\psi_{2} 
\end{aligned}\right.
\end{equation}

Moreover, we can easily get rid of $\lambda\neq 0$. Indeed, setting $$\psi_{1}(x)=\frac{1}{\vert\lambda\vert}\tilde{\psi}_{1}(x),\quad\psi_{2}(x)=\frac{\lambda}{\vert\lambda\vert^{2}}\tilde{\psi}_{2}(x),\qquad x\in\R^{2}$$
and defining $$\tilde{\beta}_{j}:= \frac{\beta_{j}}{\vert\lambda\vert^{3}},\quad j=1,2 ,$$
one ends up (dropping superscripts) with the system:
\begin{equation}\label{diracsystem}
\left\{\begin{aligned}
    (\partial_{x_{1}}+i\partial_{x_{2}})\psi_{2} &=i(2\beta_{2}\vert\psi_{1}\vert^{2}+\beta_{1}\vert\psi_{2}\vert^{2})\psi_{1} \\
         (\partial_{x_{1}}-i\partial_{x_{2}})\psi_{1} &=i(\beta_{1}\vert\psi_{1}\vert^{2}+2\beta_{2}\vert\psi_{2}\vert^{2})\psi_{2}
\end{aligned}\right.
\end{equation}
where $0<\beta_{2}\leq\beta_{1}$. 

For simplicity, we state our main result in terms of equation (\ref{diracsystem}).

\begin{thm}\label{main}
Equation (\ref{diracsystem}) admits infinitely many solutions $\psi\in C^{\infty}(\R^{2},\C^{2})$ of the form 
\be\label{form}
\psi(r,\vartheta)=\begin{pmatrix}  iu(r)e^{i\vartheta}  \\ v(r) \end{pmatrix} 
\ee

with $u,v:[0,+\infty)\longrightarrow \R$, $(r,\vartheta)$ being polar coordinates in $\R^{2}$.

 Moreover, the spinor components satisfy
\be
 u(r)v(r)>0,\qquad\forall r>0,
\ee
and there holds 
\be\label{decay}
\vert u(r)\vert \sim\frac{1}{r},\quad\vert v(r)\vert\sim\frac{1}{r^{2}}, \qquad\mbox{as}\quad r\rightarrow+\infty,
\ee
In particular,
$$ \psi\in L^{p}(\R^{2},\C^{2}),\quad \forall p>2,$$
but
$$\psi\notin L^{2}(\R^{2},\C^{2}). $$
For this reason, we say that those solutions are weakly localized.
\end{thm}
\begin{remark}
Heuristically, weak localization is expected as the $L^{2}$-spectrum of the massless Dirac operator is equal to $\R$, as it is easily seen using the Fourier transform (see \cite{diracthaller} for more details). As shown in Theorem \ref{main}, in general stationary solutions in the massless case only exhibit a polynomial decay at infinity. This is in striking contrast with the massive case, where stationary solutions (of arbitrary form) are exponentially localized (see, e.g., \cite{spectralnabilecomech} where the method of \cite{berthiergeorgescu} has been generalized to deal with nonlinear bound states in any dimensions).
\end{remark}
\begin{remark}\label{remarkscaling}
Equation (\ref{diracsystem}) is invariant by scaling. Indeed, it can be easily checked that if $\psi$ is a solution, then the same holds for the rescaled spinor
\be\label{scaling} \psi_{\delta}(\cdot):=\sqrt{\delta}\psi(\delta\cdot),\qquad \forall\delta>0.\ee
Thus it suffices to prove the existence of one (non-trivial) solution, to get multiplicity. Observe also that if $\psi$ solves the equation, then \be\label{odd}\tilde{\psi}(\cdot):=-\psi(\cdot)\ee is another solution.
\end{remark}
\begin{remark}\label{clifford}
Theorem \ref{main} is in some sense suggested by the literature on the spinorial Yamabe problem. A particular family of test spinors is used to study conformal invariants or nonlinear Dirac equations on spin manifolds (see e.g. \cite{Aubin},\cite{Isobecritical} and references therein). It is given by
\begin{equation}\label{family} 
\varphi(y)=f(y)(1-y)\cdot\varphi_{0} \qquad y\in\mathbb{R}^{2}
\end{equation} 
where $\varphi_{0}\in \mathbb{C}^{2}$, $f(y)=\frac{2}{1+\vert y\vert^{2}}$ and the dot represents the Clifford product.

It can be easily checked that they are $\mathring{H}^{\frac{1}{2}}(\R^{2},\C^{2})$-solutions to the following "isotropic" Dirac equation (corresponding to $\beta_{1}=1,\beta_{2}=\frac{1}{2}$)
\begin{equation}\label{conformal}
\mathcal{D}\varphi=\vert\varphi\vert^{2}\varphi
\end{equation} 
The spin structure of euclidean spaces is quite explicit and the spinors in (\ref{family}) can be rewritten in matrix notation as 
$$\varphi(y)=f(y)(\mathds{1}_{2}+iy_{1}\sigma_{1}+iy_{2}\sigma_{2})\cdot\varphi_{0} \qquad y\in\mathbb{R}^{2}$$
$\mathds{1}_{2}$ and $\sigma_{i}$ being the identity and the Pauli matrices, respectively. See \cite{Jost} for more details. One can show that (\ref{family}) is of the form (\ref{form}) and has the decay properties stated in Theorem \ref{main}. 
\end{remark}

The present paper is organized as follows. In (Section \ref{dynamical}) we prove (Theorem \ref{main}), exploiting a particular radial ansatz. The proof follows by direct dynamical systems arguments. Then we show in (Section \ref{variational}) that the solutions found in the first part of the paper admit a variational characterization. This is done using duality, combined with standard concentration compactness theory and Nehari manifold arguments. The last section (Section \ref{massivesection}) is devoted to some remarks concerning the massive case, where we quickly explain how the results of the present paper allow to extend those of \cite{shooting}.

\noindent\textbf{Acknowledgment.} 
I wish to thank Michael I. Weinstein for his encouragement.


\section{Existence and asymptotics}\label{dynamical}
In this section we prove Theorem (\ref{main}), providing the existence and the exact asymptotic behavior of (non-trivial) solutions of (\ref{diracsystem}) satisfying the ansatz (\ref{form}). The latter allows us to convert the PDE (\ref{diracsystem}) into a dynamical system. Indeed, passing to polar coordinates in $\R^{2}$, $(x_{1},x_{2})\mapsto ( r,\vartheta)$, the equation reads as:

\begin{equation}\label{polar}
\left\{\begin{aligned}
 -e^{i\vartheta}\left(i\partial_{r}-\frac{\partial_{\vartheta}}{r}\right)\psi_{2} &=-\left(2\beta_{2}\vert\psi_{1}\vert^{2}+\beta_{1}\vert\psi_{2}\vert^{2}\right)\psi_{1} ,
\\
-e^{-i\vartheta}\left(i\partial_{r}+\frac{\partial_{\vartheta}}{r}\right)\psi_{1} &=\left(\beta_{1}\vert\psi_{1}\vert^{2}+2\beta_{2}\vert\psi_{2}\vert^{2}\right)\psi_{2}.
\end{aligned}\right.
\end{equation}
Plugging the ansatz
\be
\psi(r,\vartheta)=\begin{pmatrix}  iu(r)e^{i\vartheta}  \\ v(r) \end{pmatrix} 
\ee
into (\ref{polar}) gives:
\begin{equation}\label{system}
\left\{\begin{aligned}
    \dot{u}+\frac{u}{r} &=v(2\beta_{2}u^{2}+\beta_{1}v^{2}) \\ 
   \dot{v}&=-u(\beta_{1}u^{2}+2\beta_{2}v^{2})
\end{aligned}\right.
\end{equation}
Thus we are lead to study the flow of the above system. 

In particular, since we are looking for localized states, we are interested in solutions to (\ref{system}) such that $$(u(r),v(r))\longrightarrow(0,0)\qquad \mbox{as} \qquad r\rightarrow +\infty$$

In order to avoid singularities and to get non-trivial solutions, we choose as initial conditions 
\be\label{initial}
u(0)=0\quad,\quad v(0)=\lambda\neq 0
\ee
Moreover, the symmetry of the system allows us to consider only the case $\lambda>0$. Thus (Theorem \ref{main}) reduces to the following

\begin{prop}\label{asymptotic}
For any $\lambda>0$ there exists a unique solution $$(u_{\lambda},v_{\lambda})\in C^{\infty}([0,+\infty),\R^{2})$$ of the Cauchy problem (\ref{system},\ref{initial}).

Moreover, there holds
\be
u_{\lambda}(r), v_{\lambda}(r)>0,\qquad\forall r>0,
\ee
and 
\be
u_{\lambda}(r)\sim\frac{1}{r},\quad v_{\lambda}(r)\sim\frac{1}{r^{2}}, \qquad\mbox{as}\quad r\rightarrow+\infty,
\ee
In particular,
$$ \psi\in L^{p}(\R^{2},\C^{2}),\quad \forall p>2,$$
but
$$\psi\notin L^{2}(\R^{2},\C^{2}). $$
\end{prop}
 \begin{figure}[!h]
        \centering
        \includegraphics[scale=.5]{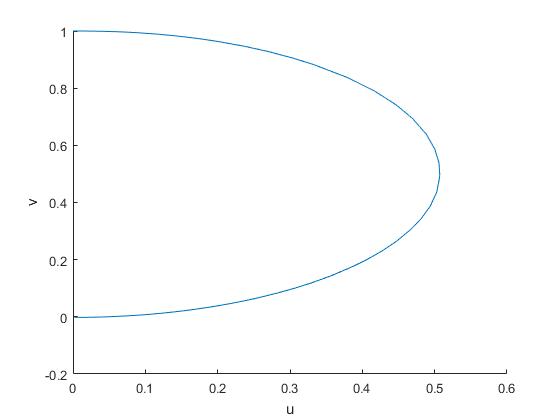}
 \caption{The trajectory of a representative solution of (\ref{system}) with $\lambda>0$.}
 \label{figura}
   \end{figure}

The proof of (Prop. \ref{asymptotic}) will be achieved in several intermediate steps.

Local existence and uniqueness of solutions of (\ref{system}) are guaranteed by the following
\begin{lemma}
\label{existence}
Let $\lambda>0$. There exist $0<R_{\lambda}\leq+\infty$ and $(u,v)\in C^{1}([0,R_{\lambda}),\mathbb{R}^{2})$ unique maximal solution to (\ref{system}), which depends continuously on $\lambda$ and uniformly on $[0,R]$ for any $0<R<R_{\lambda}$.
\end{lemma}
\begin{proof}
We can rewrite the system in integral form as
\begin{equation}\label{integral}
\left\{\begin{aligned}
   u(r) &= \frac{1}{r}\int^{r} _{0}sv(s)(2\beta_{2}u^{2}(s)+\beta_{1}v^{2}(s))ds \\ 
   v(r)&= \lambda-\int^{r}_{0}u(s)(\beta_{1}u^{2}(s)+2\beta_{2}v^{2}(s))ds
\end{aligned}\right.
\end{equation}
where the r.h.s. is a Lipschitz continuous function with $(u,v)\in C^{1}$. Then the claim follows by a contraction mapping argument, as in \cite{cv}.
\end{proof}

Given $\lambda>0$, we will denote by $(u_{\lambda},v_{\lambda})$ the corresponding (maximal) solution. Dropping the singular term in (\ref{system}) we obtain a hamiltonian system

 \begin{equation}\label{hamsyst}
\left\{\begin{aligned}
    \dot{u} &=v(2\beta_{2}u^{2}+\beta_{1}v^{2}) \\ 
   \dot{v}&=-u(\beta_{1}u^{2}+2\beta_{2}v^{2})
\end{aligned}\right.
\end{equation}
whose hamiltonian is given by

\be\label{hamiltonian}
H(u,v)=\frac{\beta_{1}}{4}(u^{4}+v^{4})+\beta_{2}u^{2}v^{2}
\ee
Consider 
\be
H_{\lambda}(r):=H(u_{\lambda}(r),v_{\lambda}(r))
\ee
then a simple computation gives
\be\label{deriveH}
\dot{H_{\lambda}}(r)=-\frac{u_{\lambda}(r)}{r}(\beta_{1}u^{2}_{\lambda}(r)+2\beta_{2}v^{2}_{\lambda}(r))\leq0
\ee
so that the energy $H$ is non-increasing along the solutions of (\ref{system}). 

This implies that $\forall r\in[0,R_{x})$, $(u_{\lambda}(r),v_{\lambda}(r))\in \{H(u,v)\leq H(0,\lambda)\}$, the latter being a compact set. Thus there holds
\begin{lemma}
Every solution to (\ref{system}) is global.
\end{lemma}
\begin{remark}
Smoothness of solutions follows by basic ODE theory.
\end{remark}
Heuristically, (\ref{system}) should reduce to (\ref{hamsyst}) in the limit $r\rightarrow+\infty$ ($u$ being bounded), that is, dropping the singular term in the first equation. The following lemma indeed shows that the solutions to (\ref{system}) are close to the hamiltonian flow (\ref{hamsyst}) as $r\rightarrow +\infty$. The proof is the same as in \cite{cv}.
\begin{lemma}
\label{stability}
Let $(f,g)$ be the solution of (\ref{hamsyst}) with initial data $(f_{0},g_{0})$. Let $(u^{0}_{n},v^{0}_{n})$ and $\rho_{n}$ be such that 
$$\rho_{n}\xrightarrow{n\rightarrow+\infty}+\infty\qquad\mbox{and}\qquad(u_{n},v_{n})\xrightarrow{n\rightarrow+\infty} (f_{0},g_{0}) $$
Consider the solution of 
$$\left\{\begin{aligned}
    \dot{u}_{n}+\frac{u_{n}}{r+\rho_{n}} &= (2\beta_{2}u^{2}_{n}+\beta_{1}v^{2}_{n})v_{n}  \\ 
   \dot{v}_{n}&= -(\beta_{1}u^{2}_{n}+2\beta_{2}v^{2}_{n})u_{n}
\end{aligned}\right. $$
such that $u_{n}(0)=u^{0}_{n}$ and $v_{n}(0)=v^{0}_{n}$.
Then $(u_{n},v_{n})$ converges to $(f,g)$ uniformly on bounded intervals.
\end{lemma}

\begin{prop}
For any $\lambda>0$, we have
\be\label{positivity}
u_{\lambda}(r), v_{\lambda}(r)>0,\qquad\forall r>0.
\ee
and
\be\label{origin}
\lim_{r\rightarrow+\infty}(u_{\lambda}(r),v_{\lambda}(r))=(0,0).
\ee
\end{prop}
\begin{proof}
Using the equations in (\ref{system}) one can compute

\be\label{first}
\frac{d}{dr}(ru_{\lambda}(r)v_{\lambda}(r))=\beta_{1}r(v^{4}_{\lambda}-u^{4}_{\lambda}),
\ee
and 

\be\label{second}
\frac{d}{dr}(r^{2}H_{\lambda}(r))=\frac{\beta_{1}}{2}r(v^{4}_{\lambda}-u^{4}_{\lambda}).
\ee
Combining (\ref{first}) and (\ref{second}) and integrating gives
\be\label{relation}
u_{\lambda}(r)v_{\lambda}(r)=2rH_{\lambda}(r)
\ee

and (\ref{positivity}) follows, $H_{\lambda}$ being positive definite.

Combining \eqref{positivity} and the second equation in \eqref{system} one sees that $\dot{v}_{\lambda}(r)\leq0$ for all $r>0$, and then
\be\label{vlimit}
\exists\lim_{r\rightarrow+\infty}v_{\lambda}(r)=:\mu\geq0.
\ee 
Moreover, since $u$ is bounded, there exists a sequecnce $r_{n}\uparrow+\infty$ such that
\be\label{usequencelimit}
\exists\lim_{n\rightarrow+\infty}u_{\lambda}(r_{n})=\delta\geq 0.
\ee 
We claim that 
\be\label{ulimit}
\lim_{r\rightarrow+\infty}u_{\lambda}(r)=\delta.
\ee
By contradiction, suppose that \eqref{ulimit} does not hold. Then there exist $\eps>0$ and another sequence $s_{n}\uparrow+\infty$ such that 
\be\label{farlimit}
\vert u_{\lambda}(s_{n})-\delta\vert\geq\eps\geq0,\qquad\forall n\in\mathbb{N}.
\ee
Up to subsequences, we can suppose that 
\be\label{usequencelimits}
\lim_{n\rightarrow+\infty}u_{\lambda}(s_{n})=\gamma\neq\delta,
\ee
for some $\gamma\geq0$. Recall that $H$ decreases along the flow of \eqref{system}, as shown in \eqref{deriveH}, and then
\be\label{hlimit}
\exists\lim_{r\rightarrow+\infty}H_{\lambda}(r)=h\geq0.
\ee
Then it follows that 
\be
(\delta,\mu),(\gamma,\mu)\in\left\{ H(u,v)=h\right\}.
\ee
It is easy to see that the algebraic equation for $u$ 
\be
H(u,\mu)=h,
\ee
has (at most) one non-negative solution and thus $\delta=\gamma$, reaching a contradiction. This proves the claim \eqref{ulimit}, and then there holds
\be
\lim_{r\rightarrow+\infty}(u_{\lambda}(r),v_{\lambda}(r))=(\delta,\mu).
\ee

Let $(\rho_{n})_{n}\subseteq\mathbb{R}$ be a sequence such that 
\be\label{rho}
\lim_{n\rightarrow+\infty}\rho_{n}=+\infty\quad,\quad\lim_{n\rightarrow+\infty}(u_{\lambda}(\rho_{n}),v_{\lambda}(\rho_{n}))=(\delta,\mu)
\ee
and consider the solution $(U,V)$ to (\ref{hamsyst}) such that $$(U(0),V(0))=(\delta,\mu). $$
By (Lemma \ref{stability}), it follows that $(u_{\lambda}(\rho_{n}+\cdot),v_{\lambda}(\rho_{n}+\cdot))$ converges uniformly to $(U,V)$ on bounded intervals. But since
\be
\lim_{n\rightarrow+\infty}(u_{\lambda}(\rho_{n}+r),v_{\lambda}(\rho_{n}+r))=(\delta,\mu),\qquad\forall r>0,
\ee
this implies that 
\be
(U(r)),V(r))=(\delta,\mu),\qquad\forall r>0
\ee
and thus $(\delta,\mu)=(0,0)$ as the latter is the only equilibrium of the hamiltonian system (\ref{hamsyst}). This proves (\ref{origin}).
\end{proof}
The above proposition shows that the solutions of (\ref{system}) actually correspond to localized solutions of the PDE (\ref{diracsystem}). The aim of the rest of the section is then to provide the exact asymptotic behavior.
\begin{prop}\label{lowergrowth}
For large $r>0$, there holds
\be\label{lower}
\frac{1}{r^{2}}\lesssim u^{2}_{\lambda}(r)+v^{2}_{\lambda}(r)\lesssim \frac{1}{r}.
\ee
\end{prop}
\begin{proof}
Remark that 
\be\label{H}
(u^{2}_{\lambda}(r)+v^{2}_{\lambda}(r))^{2}\sim H_{\lambda}(r). 
\ee
Moreover, by (\ref{hamiltonian},\ref{deriveH}) one gets 
$$ \dot{H}_{\lambda}(r)\geq-4\frac{H_{\lambda}(r)}{r},$$
and the comparison principle for ODEs implies that 
$$ H_{\lambda}(r)\gtrsim \frac{1}{r^{4}}$$
and thus by (\ref{H}), we get the first inequality in (\ref{lower}).

The second part of  (\ref{lower}) follows by (\ref{relation},\ref{H}), using the elementary inequality $$2u_{\lambda}(r)v_{\lambda}(r)\leq u^{2}_{\lambda}(r)+v^{2}_{\lambda}(r),\qquad\forall r>0.$$
\end{proof}

The first equation in (\ref{system}) can be rewritten as
\be\label{rewrite}
\frac{d}{dr}(ru_{\lambda}(r))=rv_{\lambda}(r)(2\beta_{2}u^{2}_{\lambda}(r)+\beta_{1}v^{2}_{\lambda}(r))
\ee

Since $v_{\lambda}>0$, we deduce from (\ref{rewrite}) that the function $f(r):=(ru_{\lambda}(r))$ is strictly increasing and thus
\be\label{limit}
\lim_{r\rightarrow+\infty}f(r)=:l\in(0,+\infty]
\ee
Suppose that
\be\label{divergence}
 l=+\infty.
\ee

This implies that 
\be\label{superlinear}
u_{\lambda}(r)\geq\frac{1}{r}
\ee
for $r>0$ large. Combining (\ref{superlinear}) and (\ref{lower}), using the second equation in (\ref{system}) we deduce that
\be
\dot{v}_{\lambda}(r)\lesssim-\frac{1}{r^{3}}.
\ee
Using again the comparison principle, we conclude that 
\be\label{vdecay}
v_{\lambda}(r)\lesssim\frac{1}{r^{2}}
\ee
for $r>0$ large. By (\ref{lower},\ref{vdecay}), integrating (\ref{rewrite}) gives
\be
f(r)=\int^{r}_{0}v_{\lambda}(s)\underbrace{(2\beta_{2}u^{2}_{\lambda}(s)+\beta_{1}v^{2}_{\lambda}(s))s}_{bounded}ds\lesssim\int^{+\infty}_{1}\frac{ds}{s^{2}}<+\infty,\quad \forall r>0,
\ee
thus contradicting (\ref{divergence}). Then $0<l<+\infty$, and this implies that
\be\label{ubehavior}
u_{\lambda}(r)\sim\frac{1}{r}
\ee
for large $r>0$. Since (\ref{vdecay}) holds, using the second equation in (\ref{system}) and (\ref{ubehavior}) one gets  
\be
\dot{v}_{\lambda}(r)\sim-\frac{1}{r^{3}}.
\ee
and then for large $r>0$, we have
\be
v_{\lambda}(r)\sim\frac{1}{r^{2}}.
\ee

The integrability properties of the solution follow by the fact that 
$$\vert\psi(r)\vert^{2}=u^{2}_{\lambda}(r)+v^{2}_{\lambda}(r) \sim \frac{1}{r^{2}}, $$
as $r\longrightarrow+\infty$. This concludes the proof of (Prop. \ref{asymptotic}), and thus of (Theorem \ref{main}).
\section{Variational characterization}\label{variational}
The solutions of (\ref{diracsystem}) found in the previous section by dynamical systems methods admit a variational characterization. Indeed, one can prove that they are critical points of a suitable action functional. More precisely one can show that they are \emph{least action} critical points of the corresponding action. In this sense they can be considered as \emph{ground state solutions}. Our variational argument also provides an alternative, more sophisticated, existence proof. This is not only interesting in itself, but also gives more informations about the properties of those solutions. 
\begin{remark}
The argument presented in this section works for $\mathring{H}^{\frac{1}{2}}$-solutions of (\ref{dirac}) of arbitrary form. However, we focus on symmetric solutions of the form (\ref{form}) as in that case we can also provide the exact asymptotic behavior of solutions, by the method described in the previous section.
\end{remark}
\begin{thm}\label{variationalproof}
Equation (\ref{diracsystem}) admits a family of smooth solutions in $\mathring{H}^{\frac{1}{2}}(\R^{2},\C^{2})$, of the form (\ref{form}) and satisfying the decay estimates (\ref{decay}). Moreover, they coincide with the solutions found in the previous section (Theorem \ref{main}).
\end{thm}
This section is devoted to the proof of the above theorem. Some preliminary definitions are in order.

The system (\ref{diracsystem}) can be written in a more compact form as:
\be\label{dirac}
\cD\psi=\nabla G_{\beta_{1},\beta_{2}}(\psi),
\ee
with $\psi=\begin{pmatrix}\psi_{1} \\ \psi_{2} \end{pmatrix}:\R^{2}\longrightarrow\C^{2}$, where
\be\label{nonlinearterm}
G_{\beta_{1},\beta_{2}}(\psi):=\frac{\beta_{1}}{4}(\vert\psi_{1}\vert^{4}+\vert\psi_{2}\vert^{4})+\beta_{2}\vert\psi_{1}\vert^{2}\vert\psi_{2}\vert^{2}.
\ee
To simplify notations, in the sequel we omit the indices $\beta_{j}$.

Here
\be\label{operator}
\cD:= -i(\vec{\sigma}\cdot\nabla)
\ee
is the Dirac operator and
$\vec{\sigma}\cdot\nabla:=\tilde{\sigma}_{1}\partial_{1}+\tilde{\sigma}_{2}\partial_{2}$, where
\begin{equation}\label{eq:pauli} \tilde{\sigma}_{1}:=\begin{pmatrix} 0 \quad& 1 \\ 1 \quad& 0 \end{pmatrix}\quad,\quad \tilde{\sigma}_{2}:=\begin{pmatrix} 0 \quad& i \\ -i \quad& 0 \end{pmatrix} \end{equation}
are Pauli-type matrices\footnote{We could rewrite the equation (\ref{diracsystem}) in terms of standard Pauli matrices $\sigma_{1}=\begin{pmatrix} 0 \quad& 1 \\ 1 \quad& 0 \end{pmatrix},\sigma_{2}=\begin{pmatrix} 0 \quad& -i \\ i \quad& 0 \end{pmatrix}$. This amounts to an unitary transformation on the spinor space $\C^{2}$ and does not affect our argument. However we prefer not to do so, in order to remain consistent with the notations of \cite{arbunichsparber}.}.

It is easy to see that (\ref{dirac}) is, formally, the Euler-Lagrange equation of the action functional
\be\label{action}
\cL(\psi):=\frac{1}{2}\int_{\R^{2}}\langle\psi,\cD\psi\rangle dx-\int_{\R^{2}} G(\psi)dx.
\ee
We look for critical points of (\ref{action}) belonging to the Sobolev space $\mathring{H}^{\frac{1}{2}}(\R^{2},\C^{2})$, as this is a natural choice in view of the continuous embedding
\be\label{sobolev}
\mathring{H}^{\frac{1}{2}}(\R^{2},\C^{2})\hookrightarrow L^{4}(\R^{2},\C^{2}).
\ee
given by the Gagliardo-Nirenberg inequality (see, e.g.\cite{evans}). Moreover, it is not hard to see that $\cL\in C^{1}(\mathring{H}^{\frac{1}{2}}(\R^{2},\C^{2}))$.

More precisely, we will work with the closed subspace of functions satisfying (\ref{form}):
\be\label{E}
E:=\left\{\psi\in\mathring{H}^{\frac{1}{2}}(\R^{2},\C^{2}) : \psi(r,\vartheta)=\begin{pmatrix}  iu(r)e^{i\vartheta}  \\ v(r) \end{pmatrix}, u,v:[0,+\infty)\longrightarrow\R \right\},
\ee
 $(r,\vartheta)$ being polar coordinates in $\R^{2}$. To simplify the presentation, we will sometimes adopt the notation 
\be\label{notation}
\psi=(u,v)
\ee
for $\psi\in E$, and more generally for spinors satisfying (\ref{form}). We will often identify $\psi$ with the pair $(u,v)$.


If $\psi\in E$, the action functional on $E$ reads as
\be\label{radialaction}
\mathcal{S}(u,v)=\frac{\cL(\psi)}{2\pi}=\int^{+\infty}_{0}\left(\frac{1}{2}\left(\dot{u}v+\frac{uv}{r}-u\dot{v}\right)-H(u,v)\right)rdr
\ee
where $H$ is the hamiltonian defined in (\ref{hamiltonian}).  It is not hard to see that the Euler-Lagrange equation for (\ref{radialaction}) is given by the ODE (\ref{system}).

Looking for critical points of (\ref{radialaction}) one may try to prove that it has a linking geometry ( see e.g. in \cite{es}). However, since this may not be straightforward we rather exploit the convexity of the hamiltonian $H$ in order to use duality techniques. This allows us to easily define a minimax level, the dual functional possessing a mountain pass structure. Duality is a classical tool in the study of hamiltonian systems (see \cite{criticalhamiltonian,ekeland}), which turns out to be useful also for elliptic PDEs as shown, for instance, in \cite{Isobecritical,ambrosettistruwe}. 
\begin{lemma}\label{convexity}
The function $$H:(u,v)\in\R^{2}\longrightarrow H(u,v)\in\R$$ is convex.
\end{lemma}
\begin{proof}
A simple computation gives
\be\label{hessian}
\operatorname{det}D^{2}H(u,v)=6\beta_{1}\beta_{2}(u^{4}+v^{4})+(9\beta^{1}_{2}-12\beta^{2}_{2})u^{2}v^{2}. 
\ee
Recall that $0<\beta_{2}\leq\beta_{1}$, and then by (\ref{hessian})
\be
\operatorname{det}D^{2}H(u,v)>\beta^{2}_{2}\left[6(u^{4}+v^{4})-3u^{2}v^{2} \right]>\frac{9}{2}\beta^{2}_{2}(u^{4}+v^{4})
\ee
thanks to the elementary inequality $2u^{2}v^{2}\leq u^{4}+v^{4}$, and the claim follows.
\end{proof}
 
We can thus define the \textit{Legendre transform}  $H^{*}:\R^{2}\longrightarrow\R\cup\{+\infty\}$ of $H$ as the function 
\begin{equation}\label{legendre}
H^{*}(w,z)=\sup\{\langle(u,v),(w,z)\rangle_{\R^{2}}-H(u,v)\;:\; (u,v)\in\R^{2}\} 
\end{equation}
The hamiltonian $H$ is a homogeneous polynomial of degree 4. This implies that $H^{*}$ is everywhere finite and, thanks to basic scaling properties of the Legendre transform, it is homogeneous of degree $\frac{4}{3}$. Moreover, since $H(0,0)=0$, it immediately follows from (\ref{legendre}) that $H^{*}$ is positive definite. We collect those remarks in the following
\begin{prop}\label{legendreproperties}
The function $H^{*}$ is everywhere finite, positive definite and homogeneous of degree $\frac{4}{3}$.
\end{prop}
Consider the functional, defined for $(u,v)\in L^{4}(\R_{+},rdr)^{2}$ as
\be
\cH(u,v):=\int^{+\infty}_{0}H(u,v)rdr
\ee
Its Legendre transform (or dual) is the functional \be\label{dualH}\cH^{*}:L^{\frac{4}{3}}(\R_{+},rdr)^{2}\longrightarrow \R\ee defined (with an abuse of notation) as
\be\label{legendreinfinite}
\begin{split}
\cH^{*}(w,z):&=\sup\{\langle(u,v),(w,z)\rangle_{L^{4}\times L^{\frac{4}{3}}}-\cH(u,v)\;:\; (u,v)\in L^{4}(\R_{+},rdr)^{2}\} \\&= \int^{+\infty}_{0}H^{*}(w,z)rdr
\end{split}
\ee
where $ \langle\cdot,\cdot\rangle_{L^{4}\times L^{\frac{4}{3}}}$ stands for the duality product. There holds
\be\label{dualdifferential}
d\cH\circ d\cH^{*}=\operatorname{id}_{L^{\frac{4}{3}}},\quad d\cH^{*}\circ d\cH=\operatorname{id}_{L^{4}}.
\ee

Consider the following isomorphism
\be\label{iso1}
\cD:E\longrightarrow E^{*},
\ee
and its inverse
\be\label{inverse1}
A:=\cD^{-1}:E^{*}\longrightarrow E.
\ee
where $E^{*}$ is the dual of $E$.
Let
\begin{equation}\label{radialembedding}
j:E\longrightarrow L^{4}(\R_{+},rdr)^{2}
\end{equation}
be the Sobolev embedding.
Consider the following sequence of maps 
\begin{equation}
\begin{tikzcd}
K:L^{\frac{4}{3}}(\R_{+},rdr)^{2}\arrow[r,"j^{*}"]& E^{*}\arrow[r,"A"]& E\arrow[r,"j"]& L^{4}(\R_{+},rdr)^{2},\end{tikzcd}
\end{equation}

The action functional (\ref{radialaction}) can be rewritten as
\be\label{dualityaction}
\cS(u,v)=\frac{1}{2}\langle (w,z),\cD(w,z)\rangle_{E\times E'}-\cH(j(u,v)),\qquad (u,v)\in E.
\ee
Then for $\psi=(u,v)\in E$ the differential of $\cS$ reads as
\be\label{radialdifferential}
d\cS(\psi)=\cD\psi-j^{*}d\cH(j(\psi))\in E^{*}.
\ee

We finally define the dual action functional 
\be\label{dualaction}
\begin{split}
\cS^{*}(w,z):&=\mathcal{H}^{*}(w,z)-\frac{1}{2}\langle K(w,z),(w,z)\rangle_{L^{4}\times L^{\frac{4}{3}}}\\&=\int^{+\infty}_{0}H^{*}(w,z)rdr-\frac{1}{2}\int^{+\infty}_{0}\langle K(w,z),(w,z)\rangle rdr
\end{split}
\ee

for $(w,z)\in L^{\frac{4}{3}}(\R_{+},rdr)^{2}$,which is of class $C^{1}$ on $L^{\frac{4}{3}}(\R_{+},rdr)^{2}$.

\begin{prop}\label{correspondingcritical}
There is a one-to-one correspondence between critical points of $\cS$ in $E$ and critical points of $\cS^{*}$ in $L^{\frac{4}{3}}(\R_{+},rdr)^{2}.$
\end{prop}
\begin{proof}
Let $\psi\in E$ be a critical point of $S$. Then by (\ref{radialdifferential}), we have $\cD\psi=j^{*}d\cH(j(\psi))$. Define $\varphi=d\cH(j(\psi))\in L^{\frac{4}{3}}(\R_{+},rdr)^{2}$, so that $\cD\psi=j^{*}(\varphi)$. This implies that $\psi=A\circ j^{*}(\varphi)$ and
\be\label{1}
j(\psi)=j\circ A\circ j^{*}(\varphi)=K(\varphi).
\ee
On the other hand, by (\ref{dualdifferential}) we have
\be\label{2}
j(\psi)=d\cH^{*}(\varphi).
\ee
Combining (\ref{1}) and (\ref{2}) we obtain
\be
d\cS^{*}(\varphi)=d\cH^{*}(\varphi)-K(\varphi)=0,
\ee
and then $\varphi$ is a critical point of $\cS^{*}$.

Conversely, suppose $\varphi\in L^{\frac{4}{3}}(\R_{+},rdr)^{2}$ is a critical point of $\cS^{*}$, and define $\psi=A\circ j^{*}(\varphi)\in E$. Since $\varphi$ is a critical point, we have $d\cH^{*}(\varphi)-K(\varphi)=0$. Then (\ref{dualdifferential}) implies that
\be
\varphi=d\cH\circ K(\varphi)=d\cH\circ j\circ A\circ j^{*}(\varphi)=d\cH(j(\psi)).
\ee
We have $j^{*}(\varphi)=j^{*}\circ d\cH(j(\psi))$ and $\cD\psi=j^{*}\circ d\cH(j(\psi))$, and thus $\psi$ is a critical point of $\cS$.
\end{proof}
\begin{remark}\label{rmkduality}
More generally, $\cS$ and $\cS^{*}$ have the same compactness properties and there is a one-to-one correspondence between their Palais-Smale sequences (see, e.g. \cite{Isobecritical} for more details).
\end{remark}
Since finding a critical point of $\mathcal{S}$ is equivalent to finding a critical point of the dual functional $\mathcal{S}^{*}$ we will focus on the latter, which has a simpler structure. More precisely, we will exploit the homogeneity properties of $\cS^{*}$ using a Nehari-manifold argument (see, e.g. \cite{nehari} and references therein). However, the fact that the second integral in (\ref{dualaction}) is not positive definite must be taken into account. We remark that a Nehari-type argument has been previously used by Ding and Ruf \cite{dingruf}, in the study of semiclassical states for critical Dirac equations. 


We have pointed out (Remark \ref{remarkscaling}) that the equation (\ref{dirac}), and thus the functional (\ref{action}), is scale-invariant. The same holds, of course, for the dual action $\cS^{*}$. Indeed, one can verify that it is invariant with respect to the following scaling 

\be\label{dualscaling}
\psi(\cdot)\mapsto\psi_{\delta}(\cdot):= \delta^{\frac{3}{2}}\psi(\delta \cdot),\quad \delta>0.
\ee
Moreover, even if the functional (\ref{action}) is invariant by translation, this is no longer true for (\ref{dualaction}) thanks to the ansatz (\ref{form}). Thus scaling is the only (local) symmetry which may prevent strong convergence in our variational procedure. In what follows we only sketch the rest of the proof, as it is based on standard arguments from concentration-compactness theory \cite{CCcompactI,CClimitI,struwevariational}.

First of all, it easy to see that the functional $\cS^{*}$ possesses a mountain-pass geometry.
\begin{lemma}\label{lowerboundrho}
There exists $\rho>0$ such that 
$$\alpha:=\inf\{\mathcal{S}^{*}(w,z)\;:\;(w,z)\in L^{\frac{4}{3}}(\R_{+},rdr)^{2},\;\Vert(w,z)\Vert_{\frac{4}{3}}=\rho\}>0. $$
Moreover, for $(w,z)\in L^{\frac{4}{3}}(\R_{+},rdr)^{2}$ such that $\int^{+\infty}_{0}\langle K(w,z),(w,z)\rangle rdr>0 $, there holds 
$$\lim_{t\rightarrow+\infty}\mathcal{S}^{*}(t(w,z))=-\infty. $$
\end{lemma}
\begin{proof}
Recall that the dual functional is defined as 
$$\mathcal{S}^{*}(w,z)=\int^{+\infty}_{0}H^{*}(w,z)rdr-\frac{1}{2}\int^{+\infty}_{0}\langle K(w,z),(w,z)\rangle rdr$$ 
for $(w,z)\in L^{\frac{4}{3}}(\R_{+},rdr)^{2}$. 
Since $H^{*}$ is homogeneous of degree $\frac{4}{3}$, as already remarked, the first assertion follows is $\rho>0$ is sufficiently small, the other term being quadratic. The second part of the claim follows immediately, for the same reason.
\end{proof}
In view of the above lemma it is natural to define the mountain-pass level for $\mathcal{S}^{*}$ as 
$$c:=\inf \left\{\max_{t\geq0}\mathcal{S}^{*}(t(w,z))\;:\;(w,z)\in L^{\frac{4}{3}}(\R_{+},rdr)^{2},\int^{+\infty}_{0}\langle K(w,z),(w,z)\rangle rdr>0 \right\}. $$
Remark that there holds
\be
\max_{t\geq0}\mathcal{S}^{*}(t(w,z))\geq\alpha>0,\qquad \forall (w,z)\in L^{\frac{4}{3}}(\R_{+},rdr)^{2},
\ee
where $\alpha>0$ is as in Lemma \ref{lowerboundrho}. Then we have
\be
c\geq\alpha>0.
\ee
Moreover, the homogeneity properties of the terms appearing in $\cS^{*}$ imply that  
\be\label{equality}
c=\inf_{(w,z)\in\cN}\cS^{*}(w,z)>0,
\ee
where $\cN$ is the Nehari manifold
\be\label{nehari}
\cN:=\{(w,z)\in L^{\frac{4}{3}}(\R_{+},rdr)^{2}\setminus\{0\} : \langle d\cS^{*}(w,z),(w,z)\rangle=0 \}.
\ee

We are thus led to study the minimization problem (\ref{equality}).

Let $(w_{n},z_{n})_{n\in\mathbb{N}}\subseteq\cN$ be a minimizing sequence for $\cS^{*}$. By Ekeland's variational principle (see \cite{ekeland}), we can assume that it actually is a Cerami-sequence, that is: 
\be\label{ceramidefinition} \begin{cases}
\cS^{*}(w_{n},z_{n})\longrightarrow c,\\
(1+\Vert(w_{n},z_{n}) \Vert_{L^{\frac{4}{3}}})d\cS^{*}(w_{n},z_{n})\xrightarrow{L^{4}} 0,
\end{cases}
 \qquad\mbox{as}\quad n\longrightarrow\infty.
\ee
\begin{prop}
The sequence $(w_{n},z_{n})_{n\in\mathbb{N}}\subseteq\cN$ is bounded in $L^{\frac{4}{3}}(\R_{+},rdr)^{2}$.
\end{prop}
\begin{proof}
We have $$d\cS^{*}(w_{n},z_{n})=\nabla H^{*}(w_{n},w_{n})-K(w_{n},z_{n}).$$ Since the function $H^{*}$ is $\frac{4}{3}$-homogeneous there holds $$ \langle\nabla H^{*}(w_{n},w_{n}),(w_{n},z_{n})\rangle=\frac{4}{3}H^{*}(w_{n},z_{n}).$$
Then by the definition of the Nehari manifold (\ref{nehari}), it follows that \be\label{norm}\cS^{*}(w_{n},z_{n})=\frac{1}{3}\int^{+\infty}_{0}H^{*}(w_{n},z_{n})rdr.\ee

The claim thus follows because 
\be\label{equivalentnorm}
 \int^{+\infty}_{0}H^{*}(w_{n},z_{n})rdr \sim \Vert(w_{n},z_{n})\Vert^{\frac{4}{3}}_{L^{\frac{4}{3}}}, 
 \ee
and $(w_{n},z_{n})_{n\in\mathbb{N}}$ is a minimizing sequence.
\end{proof}
By the above lemma we may assume that 
\be\label{weak}
(w_{n},z_{n})\rightharpoondown (w,z), \qquad\mbox{weakly in}\quad L^{\frac{4}{3}}(\R_{+},rdr)^{2},
\ee
as $n\rightarrow+\infty$. One needs to study the concentration behavior of the minimizing sequence in order to prove \emph{strong} $L^{\frac{4}{3}}$-convergence. 

We already remarked that scaling invariance may prevent strong convergence, as Cerami sequences may blow-up around some points. Since we are essentially working with radial functions, concentration may only occur at the origin. More precisely (recall \eqref{equivalentnorm}), there holds
\be\label{weakmeasure}
H^{*}(w_{n},z_{n})rdr=:\nu_{n}\rightharpoondown\nu:= H^{*}(w,z)rdr+\alpha_{0}\delta_{0},
\ee
weakly in the sense of measures, where $\delta_{0}$ is a Dirac mass concentrated at the origin and $\alpha_{0}\geq0$.

Recall that 
\be\label{totalmass}
 \int^{+\infty}_{0}H^{*}(w_{n},z_{n})rdr =3\cS^{*}(w_{n},z_{n})\rightarrow 3c,\qquad\mbox{as}\quad n\rightarrow+\infty.
\ee
Suppose that the minimizing sequence $(w_{n},z_{n})_{n\in\mathbb{N}}$ splits into two bumps, one of them centered around the origin and the other one carrying a positive part of the "mass" at infinity (the \emph{dichotomy case} \cite{CCcompactI}). More precisley, assume that there exist $0<b<3c$, and two sequences of radii $r_{n},r'_{n}\rightarrow+\infty$, with $\frac{r_{n}}{r'_{n}}\rightarrow0$ such that
\be\label{dicotomia}
\int^{r_{n}}_{0}H^{*}(w_{n},z_{n})rdr\rightarrow b,\qquad \int^{r'_{n}}_{r_{n}}H^{*}(w_{n},z_{n})rdr\rightarrow 0,\qquad\mbox{as}\quad n\rightarrow+\infty.
\ee
Take a cutoff function $\theta\in C^{\infty}_{c}([0,\infty))$, $0\leq\theta\leq 1$ such that $\theta\equiv 1$ on $[0,1]$ and $\theta\equiv 0$ on $[2,\infty)$, and define
\be\label{bumps}
(w^{1}_{n},z^{1}_{n})(r):=\theta\left(\frac{r}{r_{n}} \right)(w_{n},z_{n}),\qquad (w^{2}_{n},z^{2}_{n}):=\left(1-\theta\left(\frac{r}{r'_{n}}\right) \right)(w_{n},z_{n})(r).
\ee
There holds $$\cS^{*}(w_{n},z_{n})-\cS^{*}(w^{1}_{n},z^{1}_{n})-\cS^{*}(w^{2}_{n},z^{2}_{n})\rightarrow0, \qquad\mbox{as}\quad n\rightarrow+\infty, $$
and both sequences in \eqref{bumps} are Cerami sequences for the functional $\cS^{*}$, that is
\be
0<\cS^{*}(w^{k}_{n},z^{k}_{n})\rightarrow c_{k}<c,
\ee
and
\be\label{cerami2}
\left(1+\Vert(w^{k}_{n},z^{k}_{n})\Vert_{L^{\frac{4}{3}}}\right)d\cS^{*}(w^{k}_{n},z^{k}_{n})\xrightarrow{L^{4}}0, 
\ee
as $n\rightarrow+\infty$, with $k=1,2$.
\begin{remark}
The above estimates can be worked out (along the same lines as in \cite[Section 2.1]{CClimitII}) recalling that the operator $K$ in \eqref{dualaction} acts as $\cD^{-1}$ and exploiting the decay of the corresponding Green kernel $$G(x,y)=-\frac{1}{2\pi}\frac{x-y}{\vert x-y\vert^{2}}\cdot,$$ where the dot indicates the Clifford product (see Remark \ref{clifford}). 
\end{remark}
Consider, for instance, the sequence $(w^{1}_{n},z^{1}_{n})_{n\in\mathbb{N}}$. Then the condition \eqref{cerami2} implies that
\be\label{asymptoticallynehari}
\langle d\cS^{*}(w^{1}_{n},z^{1}_{n}),(w^{1}_{n},z^{1}_{n}) \rangle_{L^{4}\times L^{\frac{4}{3}}}\longrightarrow0,\qquad\mbox{as}\quad n\rightarrow+\infty,
\ee 
that is, the sequence is asymptotically on the Nehari manifold $\cN$. 

Moreover, there exists a sequence $t_{n}>1$ such that $t_{n}(w^{1}_{n},z^{1}_{n})\in\cN,\forall n\in\mathbb{N}$ ( see, e.g. \cite{nehari} ). Recalling that $H^{*}$ is $\frac{4}{3}$-homogeneous, the definition \eqref{nehari} of $\cN$ then gives
\be\label{scaled}
\frac{4}{3t^{\frac{2}{3}}_{n}}\int^{\infty}_{0}H^{*}(w^{1}_{n},z^{1}_{n})rdr=\int^{\infty}_{0}\langle K(w^{1}_{n},z^{1}_{n}),(w^{1}_{n},z^{1}_{n})\rangle rdr,\qquad\forall n\in\mathbb{N}.
\ee
Combining \eqref{asymptoticallynehari} and \eqref{scaled} one gets
\be
\left(1-t^{-\frac{2}{3}}_{n} \right)\frac{4}{3}\int^{\infty}_{0}H^{*}(w^{1}_{n},z^{1}_{n})rdr\longrightarrow 0,\qquad\mbox{as}\quad n\rightarrow+\infty,
\ee
and thus the first condition in \eqref{dicotomia} implies that
\be\label{tendeuno}
\lim_{n\rightarrow+\infty}t_{n}=1.
\ee
Then since $t_{n}(w^{1}_{n},z^{1}_{n})\in\cN$, by \eqref{dicotomia} and \eqref{tendeuno} there holds $$\cS^{*}(t_{n}(w^{2}_{n},z^{2}_{n}))=\frac{t^{\frac{4}{3}}_{n}}{3}\int^{+\infty}_{0}H^{*}(w^{1}_{n},z^{1}_{n})rdr\in(0,c),$$
for $n$ large, contradicting the minimality of $c=\inf_{\cN}\cS^{*}$.  This means that dichotomy cannot occur, and then the sequence of measures $(d\nu_{n})_{n\in\mathbb{N}}$ in \eqref{weakmeasure} is \emph{tight}. Consequently, up to extraction, \eqref{totalmass} gives
\be\label{limitmass}
\int_{\R^{2}}d\nu=3c.
\ee

Up to suitably rescaling the sequence, we may assume that the weak limit is non-trivial, that is, \emph{vanishing} is also excluded \cite{CCcompactI}. Indeed, one can find a sequence $\lambda_{n}>0, n\in\mathbb{N}$ such that for the rescaled spinor 
\be\label{scalenormalization}
(\tilde{w}_{n},\tilde{z}_{n})(\cdot):=\lambda^{\frac{3}{2}}_{n}(w_{n},z_{n})(\lambda_{n}\cdot)
\ee
there holds 
\be\label{normalization}
Q_{n}(1)=\int^{1}_{0}H^{*}(\tilde{w}_{n},\tilde{z}_{n})rdr=c,\qquad\forall n\in\mathbb{N},
\ee
where $Q_{n}(\cdot)$ is the concentration function of $(\tilde{w}_{n},\tilde{z}_{n})$ \cite{CCcompactI,CClimitI}.

Assume that $(w,z)=(0,0)$. Then by \eqref{weakmeasure} we have $\nu=\alpha_{0}\delta_{0}$, and the normalization \eqref{normalization} gives $\alpha_{0}\leq c$. Then \eqref{weakmeasure} and \eqref{limitmass} imply
\be
3c=\int_{\R^{2}}d\nu=\alpha_{0}\leq c,
\ee
which is clearly absurd. This allows us to conclude that 
\be
(w,z)\neq(0,0),
\ee
that is, the above normalization \eqref{normalization} rules out the \emph{vanishing case}. 

The last step in order to conclude the strong convergence of the minimizing sequence $(w_{n},z_{n})_{n\in\mathbb{N}}$ is to show that actually $\alpha_{0}=0$ in \eqref{weakmeasure}. If this is not the case, since by \eqref{normalization} there holds $0<\alpha_{0}\leq c$, this property and the tightness of the sequence $(d\nu_{n})_{n\in\mathbb{N}}$ imply that the sequence $(w_{n},z_{n})_{n\in\mathbb{N}}$ splits into two parts, one blowing up at the origin, as $n\rightarrow+\infty$, and concentrating a portion $\alpha_{0}$ of the mass at that point, and another non-trivial part carrying the rest of the mass, essentially localized in an interval of the form $[1,R]$ (corresponding to an annulus in $\R^{2}$), for some $R>0$. Exploiting again the scale-invariance of the problem, one can suitably rescale the sequence, as in \eqref{scalenormalization}, removing the blowup at the origin, and at the same time "sending at infinity" the bump localized in $[1,R]$. In this way we have created a sequence for which \emph{dichotomy holds} (see \eqref{dicotomia}), and this is not possible, as already shown.

Finally, we conclude that $\alpha_{0}=0$ in \eqref{weakmeasure} and then 
\be\label{strongg}
(w_{n},z_{n})\xrightarrow{L^{\frac{4}{3}}}(w,z)\in\cN\qquad\mbox{as}\quad n\rightarrow+\infty.
\ee
Thus $\cS^{*}(w,z)=\min_{\cN}\cS^{*}$, and correspondingly 

 $$(u,v)=\left(A\circ j^{*} \right)(w,z),$$ is a critical point of $\cS$.

\begin{remark}
Since the Nehari manifold $\cN$ contains all critical points of $\cS^{*}$ and $\cS^{*}(w,z)=\min_{\cN}\cS^{*}$, we conclude that $(w,z)$ is a \emph{least action} critical point of $\cS^{*}$. In this sense it can be considered a sort of \emph{ground state}. The same remark holds for $(u,v)$, as a critical point of $\cS$.
\end{remark}

Since we are dealing with a critical equation, smoothness of solutions is not authomatic as standard bootstrap arguments do not apply. Anyway, the regularity result proven in \cite{remarkdirac} (which holds for weak solutions in $L^{4}$) ensures that $(u,v)$ actually is of class $C^{\infty}$. Being a critical point of $\cS$, $(u,v)$ solves the Euler-Lagrange equation (\ref{system}), and smoothness forces \be u(0)=0,\ee as we cannot have singularities. Moreover, since $(u,v)$ is a non-trivial solution, necessarily \be v(0)\neq0.\ee Assume, for instance, $v(0)=\lambda>0$. Since the equation is scale-invariant and odd, as anticipated in Remark (\ref{remarkscaling}), we get a continuous family of (non-trivial) solutions $(u_{\lambda},v_{\lambda})$ parametrized by $\lambda\neq0$, applying those symmetries. Uniqueness for (\ref{system}) then allows to conclude the proof. 



\section{The massive case}\label{massivesection}
Dirac points are protected by particular symmetries \cite{introdiracmaterials}, and this is the case for $\mathcal{P}\mathcal{T}$-symmetry (parity + time inversion) for honeycomb Schr\"{o}dinger operators. Indeed, in \cite{binding,FWhoneycomb} it is proved that a suitable perturbation breaking such symmetry creates a gap, thus lifting the conical degeneracy. This results in a mass term in the effective Dirac operator (\ref{operator})\footnote{As already remarked, one can express everything in term of standard Pauli matrices by exchanging spinor components.}: 

\be\label{massiveoperator}
\cD_{m}:= -i\vec{\sigma}\cdot\nabla +m\tilde{\sigma}_{3},\qquad \tilde{\sigma}_{3}:=\begin{pmatrix} -1 \quad& 0 \\ 0 \quad& 1 \end{pmatrix},
\ee
as it can be seen (at least formally) using a multiscale expansion analogous to the one done in \cite{FTWprotected,ilanweinstein}.
In \cite{shooting} we proved the existence of solitary waves for the massive variant of (\ref{diracsystem}), with a pure cubic nonlinearity, corresponding to the choice $$\left(\beta_{1}=1, \beta_{2}=\frac{1}{2}\right).$$ This particular case is a model example as it exhibits all the analytical difficulties, being $H^{\frac{1}{2}}$-\textit{critical}. Our proof in \cite{shooting} takes advantages of the ansatz (\ref{form}) in order to solve compactness issues, and it is based on dynamical systems techniques. More precisely, after a suitable rescaling we exploit the properties of the explicit solution to the (massless) limiting problem. The analysis carried out in the present paper allows us to deal with the general case: $$0<\beta_{2}\leq\beta_{1},$$
and thus with the equation
$$
i\partial_{t}\alpha+\cD_{m}\alpha=\nabla G_{\beta_{1},\beta_{2}}(\alpha),
$$
with $G_{\beta_{1},\beta_{2}}$ is defined as in (\ref{nonlinearterm}). Looking for stationary solutions
$$\alpha(t,x)=e^{i\omega t}\psi(x),\qquad (t,x)\in\R\times\R^{2},$$
gives to the following equation for $\psi(\cdot)$
\be\label{generalcase}
(\cD_{m}-\omega)\psi=\nabla G_{\beta_{1},\beta_{2}}(\psi)
\ee
 where the frequency lies in the spectral gap $\omega\in(-m,m)$, as the spectrum of $\cD_{m}$ is given by
$$\sigma(\cD_{m})=(-\infty,-m]\cup[m,+\infty),$$
see e.g. \cite{diracthaller}.
 
Exploiting the ansatz (\ref{form}) one is lead to study the flow of the following system
\begin{equation}\label{massivesystem}
\left\{\begin{aligned}
    \dot{u}+\frac{u}{r} &=(2\beta_{2}u^{2}+\beta_{1}v^{2})v-(m-\omega)v \\ 
   \dot{v}&=-(\beta_{1}u^{2}+2\beta_{2}v^{2})u-(m+\omega)u
\end{aligned}\right.
\end{equation}
complemented with initial data $u(0)=0,v(0)=\lambda\neq0$.

The following result holds:
 \begin{thm}\label{massive}
For each $\omega\in(-m,m)$, equation (\ref{generalcase}) admits a smooth localized solution, with exponential decay at infinity. 
\end{thm}
\begin{remark}
Equation (\ref{generalcase}) is odd and thus there actually are two distinct solutions, due to this symmetry property. Notice that the mass term in (\ref{massiveoperator}) breaks the conformal invariance of the equation, so that the existence of infinitely many solutions does not follow authomatically in this case. This still is an open problem, to our knowledge.
\end{remark}
The theorem can be proved studying (\ref{massivesystem}) thanks to the shooting argument given in \cite{shooting}. The delicate part consists in controlling the error committed in approximating the solution with that of the limiting problem \eqref{system} on a suitable interval. However, thanks to the asymptotic behavior (\ref{decay}), the proof follows with minor modifications (see Section 2.2 and Appendix in \cite{shooting}). 

\bibliographystyle{siam}
\bibliography{MasslessDirac}

\end{document}